\documentclass[12pt,notitlepage,reqno]{amsart}

\usepackage{amssymb,amsmath,amsthm,mathrsfs,xspace,multicol,stmaryrd}
\usepackage[mathscr]{eucal}
\usepackage{amscd}
\usepackage[breaklinks=true]{hyperref}%To get clickable links to papers
\usepackage{url}

\newcommand{\blankbrac}{\left \{ \cdot,\cdot \right \}}
\newcommand{\brac}[2]{\left \{ #1,#2 \right\}}

\newcommand{\cbrac}[2]{\left \llbracket #1,#2 \right \rrbracket}

\newcommand{\ham}[1]{\Omega^{#1}_{\mathrm{Ham}}\left(M\right)}

\newcommand{\cinf}{C^{\infty}}

\renewcommand{\L}{\mathcal{L}}

\newcommand{\ip}[1]{\iota_{v_{#1}}}

\newcommand{\alphak}[1]{\alpha_{1} \tensor \cdots \tensor \alpha_{#1}}
\newcommand{\alphadk}[1]{\alpha_{1},\hdots,\alpha_{#1}}
\newcommand{\alphask}[1]{\alpha_{\sigma(1)} \tensor \cdots \tensor \alpha_{\sigma(#1)}}
\newcommand{\alphasdk}[1]{\alpha_{\sigma(1)}, \hdots,\alpha_{\sigma(#1)}}
\newcommand{\vk}[1]{v_{\alpha_{1}} \wedge \cdots \wedge  v_{\alpha_{#1}}}

\newcommand{\LX}{\mathfrak{X}^{\wedge {\bullet}}}

\newcommand{\VectH}{\mathfrak{X}_{\mathrm{Ham}}}

\newcommand{\R}{\mathbb{R}}

\newcommand{\innerprod}[2]{\langle #1,#2 \rangle}
\newcommand{\g}{\mathfrak{g}}

\newcommand{\tensor}{\otimes}

\newcommand{\maps}{\colon}

\newcommand{\Sh}{\mathrm{Sh}}

\newcommand{\Sn}{\mathcal{S}}
\renewcommand{\deg}[1]{\left \lvert #1 \right \rvert}
\newcommand{\X}{\mathfrak{X}}

\newcommand{\Lie}{L_{\infty}}

\DeclareMathOperator{\id}{\mathrm{id}}

\DeclareMathOperator{\Leib}{\mathrm{Leib}}

\theoremstyle{plain}
\newtheorem{theorem}{Theorem}[section]
\newtheorem{prop}[theorem]{Proposition}
\newtheorem{lemma}[theorem]{Lemma}

\newtheorem{definition}[theorem]{Definition}

\theoremstyle{remark}

\title[$L_{\infty}$-algebras from multisymplectic geometry]
{$L_{\infty}$-algebras from multisymplectic geometry}\author{Christopher L.\ Rogers}
\email{\texttt{chris@math.ucr.edu}} \address{Department of
  Mathematics, University of California, Riverside, California 92521,
  USA} \date{\today} \thanks{This work was partially supported by a
  grant from The Foundational Questions Institute.}
\begin{document}

\begin{abstract}
  A manifold is multisymplectic, or more specifically $n$-plectic, if
  it is equipped with a closed nondegenerate differential form of
  degree $n+1$.  In previous work with Baez and Hoffnung, we
  described how the `higher analogs' of the algebraic and geometric
  structures found in symplectic geometry should naturally arise in
  2-plectic geometry. In particular, just as a
  symplectic manifold gives a Poisson algebra of functions, any
  2-plectic manifold gives a Lie 2-algebra of 1-forms and
  functions. Lie $n$-algebras are examples of $L_{\infty}$-algebras:
  graded vector spaces equipped with a collection of
  skew-symmetric multi-brackets that satisfy a generalized Jacobi
  identity. Here, we generalize our previous result.  Given an
  $n$-plectic manifold, we explicitly construct a corresponding Lie
  $n$-algebra on a complex consisting of differential forms whose
  multi-brackets are specified by the $n$-plectic structure. We also
  show that any $n$-plectic manifold gives rise to another kind of
  algebraic structure known as a differential graded Leibniz
  algebra. We conclude by describing the similarities between these
  two structures within the context of an open problem in the theory
  of strongly homotopy algebras. We also mention a possible connection
  with the work of Barnich, Fulp, Lada, and Stasheff
  on the Gelfand-Dickey-Dorfman formalism.
 \end{abstract}

\maketitle

\section{Introduction}
Multisymplectic manifolds are smooth manifolds equipped with a closed,
nondegenerate differential form. In this paper, we call such a manifold
`$n$-plectic' if the form has degree $n+1$. Hence
a 1-plectic manifold is a symplectic manifold. Multisymplectic
geometry originated in covariant Hamiltonian formalisms for classical
field theory, just as symplectic geometry originated in classical
mechanics. (See, for example, \cite{Carinena-Crampin-Ibort,GIMM,
  Helein:2002wf,Kijowski:1973gi}, as well as the review article
\cite{RomanRoy:2005en}.)  More specifically, in $(n+1)$-dimensional
classical field theory, one can construct a finite-dimensional
$(n+1)$-plectic manifold known as a `multi-phase space'.  Particular
submanifolds of this space correspond to solutions of the theory. The
data encoded by the submanifolds include the value of the field as
well as the value of its `multi-momentum' at each point in space-time.
The multi-momentum is a quantity that is related to the time and
spacial derivatives of the field via a Legendre transform, in a manner
similar to the relationship between the velocity of a point particle
and its momentum. In fact, a $(0+1)$-dimensional theory is just the
classical mechanics of point particles, and the corresponding
1-plectic manifold is the usual extended phase space whose points
correspond to time, position, energy, and momentum.  

However, multisymplectic manifolds can be found outside
the context of classical field theory and are interesting from a
purely geometric point of view. For motivation, we provide the
following examples:
\begin{itemize}
\item{An $(n+1)$-dimensional orientable manifold
      equipped with a volume form is an $n$-plectic manifold.}
\item{Given a manifold $M$, the $n$-th exterior power of the
      cotangent bundle $\Lambda^{n}T^{\ast}M$ admits a canonical
      closed non-degenerate $(n+1)$-form and therefore is an
      $n$-plectic manifold. This is a generalization of the canonical
      symplectic structure on the cotangent bundle.} 
\item{Any compact simple Lie group $G$ is a 2-plectic manifold 
when equipped with the canonical bi-invariant 3-form
\[
 \nu(x,y,z)=\innerprod{x}{[y,z]}, 
\]
where $x,y,z \in \g$ and $\innerprod{\cdot}{\cdot}$ is the Killing
form.  The relationship between this 2-plectic manifold and the
topological group $\mathrm{String}(n)$, which arises in the study of
spin structures on loop spaces, can be found in our previous work with
Baez \cite{Baez:2009uu}.  }
\item{Let $(M,g)$ be a Riemannian manifold which admits two anti-commuting, almost
    complex structures $J_{1},J_{2} \maps TM \to TM$, i.e.\ $J_{1}^2 =
    J_{2}^2=-\id$ and $J_{1}J_{2}=-J_{2}J_{1}$. Then
    $J_{3}=J_{1}J_{2}$ is also an almost complex structure. If $J_{1},J_{2},J_{3}$
    preserve the metric $g$, then one can define the 2-forms $\theta_{1},\theta_{2},\theta_{3}$,
    where $\theta_{i}(v_{1},v_{2})=g(v_{1},J_{i}v_{2})$. If each $\theta_{i}$ is closed,
    then $M$ is called a hyper-K\"{a}hler manifold \cite{Swann:1991}. 
    Given such a manifold, one can construct the 4-form: 
   \[
   \omega=\theta_{1}\wedge \theta_{1} + \theta_{2}\wedge \theta_{2} + \theta_{3}\wedge \theta_{3}.
   \]
   It is straightforward to show $\omega$ is closed and
   nondegenerate. Hence a hyper-K\"{a}hler manifold is a 3-plectic
   manifold \cite{Cantrijn:1999et}.
} 
\end{itemize}
More examples, as well as the multisymplectic analogs of isotropic submanifolds,
co-isotropic submanifolds and real polarizations can be found in the
papers by Cantrijn, Ibort, and de Le\'{o}n \cite{Cantrijn:1999et} and
Ibort \cite{Ibort:2000}.

In our previous work with Baez and Hoffnung \cite{Baez:2008bu}, we
described how $2$-plectic geometry can be understood as higher or
`categorified' symplectic geometry. 
For example, if a symplectic structure is
integral, then it corresponds to the curvature of a principal
$U(1)$-bundle. Similarly, in the 2-plectic case, the
integrality condition implies that the 2-plectic form is
the curvature of a $U(1)$-gerbe, the higher analog of a principal
$U(1)$-bundle. Just as a principal bundle can be described as a
certain kind of sheaf (its sheaf of sections), a gerbe can be
described as a certain kind of categorified sheaf or stack.

From the algebraic point of view, the fundamental object in symplectic
geometry is the Poisson algebra of smooth functions whose bracket is
induced by the symplectic form. On a 2-plectic manifold, we showed
that a 2-plectic structure gives rise to a Lie 2-algebra on a chain
complex consisting of smooth functions and certain 1-forms which we
call Hamiltonian \cite{Baez:2008bu}.  Lie $n$-algebras (equivalently,
$n$-term $L_{\infty}$-algebras) are higher analogs of differential
graded Lie algebras. They consist of a graded vector space
concentrated in degrees $0,\ldots,n-1$ and are equipped
with a collection of skew-symmetric $k$-ary brackets, for $1 \leq k
\leq n+1$, that satisfy a generalized Jacobi identity
\cite{Lada-Markl,LS}. In particular, the $k=2$ bilinear bracket behaves
like a Lie bracket that only satisfies the ordinary Jacobi identity up
to higher coherent homotopy.

One example of a 2-plectic manifold is the multi-phase space for the
classical bosonic string \cite{GIMM}. We emphasize that this space is 
finite-dimensional, and should not be confused with the infinite-dimensional symplectic
manifold that is used as a phase space in string field theory
\cite{Bowick-Rajeev,Merkulov:1992xm}. Just as the Poisson algebra of smooth
functions represents the observables of a system of particles, we
showed that the Lie 2-algebra of Hamiltonian 1-forms contains the
observables of the bosonic string \cite{Baez:2008bu}. 

We should mention that there exists other geometric objects, such as Courant
algebroids, that also behave like higher symplectic manifolds \cite{Liu:1997}. Interestingly,
Courant algebroids and 2-plectic manifolds have several features in
common. In particular, string theory, closed 3-forms and Lie 2-algebras all play
important roles in the theory of Courant algebroids \cite{Roytenberg-Weinstein,Severa1}. 
We have discussed some details of the relationship between Courant
algebroids and 2-plectic manifolds elsewhere \cite{Rogers:2010ac}. See
also Zambon's recent work \cite{Zambon:2010ka} which relates 2-plectic
geometry to higher Dirac structures.

In the present work, we generalize our previous result
\cite{Baez:2008bu} involving 2-plectic manifolds and Lie 2-algebras to
$n$-plectic manifolds for arbitrary $n \geq 1$.  Given an $n$-plectic
manifold, we define a particular space of $(n-1)$-forms as
Hamiltonian, and explicitly construct a Lie $n$-algebra on a complex
consisting of these forms and arbitrary $p$-forms for $0 \leq p \leq
n-2.$ The bilinear bracket, as well as all higher $k$-ary brackets, are
specified by the $n$-plectic structure.  We then show that any
$n$-plectic manifold gives rise to another kind of algebraic structure
known as a differential graded (dg) Leibniz algebra. A dg Leibniz
algebra is a graded vector space equipped with a degree $-1$ differential
and a bilinear bracket that satisfies a Jacobi-like identity, but does
not need to be skew-symmetric.  There is an interesting relationship between
the bilinear bracket on the Lie $n$-algebra and the bracket on the
corresponding dg Leibniz algebra. We describe some similarities
between these two structures within the context of an open problem in
the theory of strongly homotopy algebras. Finally, we point out that
Barnich, Fulp, Lada, and Stasheff have shown that
$L_{\infty}$-algebras naturally arise in the Gelfand-Dickey-Dorfman
formalism for classical field theory \cite{Barnich:1997ij}, and that
recent work by Bridges, Hydon, and Lawson \cite{Bridges:2010}
relating multisymplectic geometry to the variational bicomplex may possibly be
used to study the similarities between these $L_{\infty}$-algebras and the Lie $n$-algebras
constructed here. 

\section{Notation and preliminaries}
\subsection{Graded linear algebra}
Let $V$ be a graded vector space. 
Let $x_{1},\hdots,x_{n}$ be elements of $V$ and $\sigma \in \Sn_n$ a permutation. The \textbf{Koszul sign} $\epsilon(\sigma)=\epsilon(\sigma ; x_{1},\hdots,x_{n})$ is defined by the equality
\[
x_{1} \wedge \cdots \wedge x_{n} = \epsilon(\sigma ; x_{1},\hdots,x_{n}) x_{\sigma(1)} \wedge \cdots \wedge x_{\sigma(n)}
\]
which holds in the free graded commutative algebra generated by
$V$. Given $\sigma \in \Sn_n$, let $(-1)^{\sigma}$
denote the usual sign of a permutation. Note that $\epsilon(\sigma)$ does
not include the sign $(-1)^{\sigma}$.  

We say $\sigma \in \Sn_{p+q}$ is a {\bf $\mathbf{(p,q)}$-unshuffle}
iff $\sigma(i) < \sigma(i+1)$ whenever $i \neq p$.  The set of
$(p,q)$-unshuffles is denoted by $\Sh(p,q)$. For example, $\Sh(2,1) =
\{ (1), (23), (123) \}$.

If $V$ and $W$ are graded
vector spaces, a linear map $f \maps V^{\tensor n} \to W$ is
\textbf{skew-symmetric} iff
\[
f(v_{\sigma(1)},\hdots,v_{\sigma(n)}) = (-1)^{\sigma}\epsilon(\sigma)
f(v_{1},\hdots,v_{n}),
\]
for all $\sigma \in \Sn_{n}$. The degree of an element $x_{1} \tensor \cdots
\tensor x_{n} \in V^{\tensor \bullet}$ of the graded tensor algebra
generated by $V$ is defined to be $\deg{x_{1} \tensor \cdots
\tensor x_{n}}=\sum_{i=1}^{n} \deg{x_{i}}$.

\subsection{Multivector calculus} In order to aid our computations, we
introduce some notation and review the Cartan calculus involving multivector
fields and differential forms. We follow the notation and sign
conventions found in Appendix A of the paper by Forger, Paufler, and
R\"{o}mer \cite{Forger:2002ak}. Let $\X(M)$ be the $\cinf(M)$-module of
vector fields on a manifold $M$ and
let 
\[
\LX(M)=\bigoplus^{\dim M}_{k=0}\Lambda^{k} \left(\X(M) \right)
\]
be the graded commutative algebra of multivector fields. On $\LX(M)$ there is a
$\R$-bilinear map $[\cdot,\cdot] \maps \LX(M) \times \LX(M) \to \LX(M)$ called
the \textbf{Schouten bracket}, which gives $\LX(M)$ the structure of a
Gerstenhaber algebra. This means the Schouten bracket is a
degree $-1$ Lie bracket which satisfies the graded Leibniz rule with respect to the
wedge product. The Schouten bracket of two decomposable multivector fields
$u_{1} \wedge \cdots \wedge u_{m}, v_{1} \wedge \cdots \wedge v_{n}
\in \LX(M)$ is
\begin{multline} \label{Schouten}
\left [ u_{1} \wedge \cdots \wedge u_{m}, v_{1} \wedge \cdots \wedge
  v_{n} \right] 
= \\\sum_{i=1}^{m} \sum_{j=1}^{n} (-1)^{i+j} [u_{i},v_{j}]
\wedge u_{1} \wedge \cdots \wedge \hat{u}_{i} \wedge  \cdots \wedge
u_{m}\\
\quad \wedge v_{1} \wedge \cdots \wedge \hat{v}_{j} \wedge \cdots \wedge v_{n},
\end{multline}

where $[u_{i},v_{j}]$ is the usual Lie bracket of vector fields.

Given a form $\alpha \in \Omega^{\bullet}(M)$, the \textbf{interior product} of a decomposable
multivector field $v_{1} \wedge \cdots \wedge v_{n}$ with $\alpha$ is
\begin{equation} \label{interior}
\iota(v_{1} \wedge \cdots \wedge v_{n}) \alpha = \iota_{v_{n}} \cdots
\iota_{v_{1}} \alpha,
\end{equation}
where $\iota_{v_{i}} \alpha$ is the usual interior product of vector
fields and differential forms. The interior product of an arbitrary
multivector field is obtained by extending the above formula by $\cinf(M)$-linearity. 

The \textbf{Lie derivative} $\L_{v}$ of a differential form along a multivector field $v \in
\LX(M)$ is defined via the graded commutator of $d$ and $\iota(v)$:
\begin{equation} \label{Lie}
\L_{v} \alpha =  d \iota(v) \alpha - (-1)^{\deg{v}} \iota(v) d\alpha,
\end{equation}
where $\iota(v)$ is considered as a degree $-\deg{v}$ operator.

The last identity we will need involving multivector fields is for the graded commutator of
the Lie derivative and the interior product. Given $u,v \in
\LX(M)$, it follows from Proposition A3 in \cite{Forger:2002ak} that
\begin{equation} \label{commutator}
\iota([u,v]) \alpha = (-1)^{(\deg{u}-1)\deg{v}} \L_{u} \iota(v)  \alpha - \iota(v)\L_{u} \alpha.
\end{equation}

\section{Multisymplectic geometry} \label{multisymplectic}
We use the definition of a multisymplectic form given by Cantrijn, Ibort, and de Le\'{o}n \cite{Cantrijn:1999et}.
Many of the definitions and basic results for $n$-plectic
structures presented in this section appeared previously in our work with Baez and Hoffnung
\cite{Baez:2008bu}. 
\begin{definition}[\cite{Baez:2008bu,Cantrijn:1999et}]
\label{multisymp}
An $(n+1)$-form $\omega$ on a smooth manifold $M$ is 
{\bf multisymplectic}, or more specifically
an {\boldmath $n$}-{\bf plectic structure}, if it is both closed:
\[
    d\omega=0,
\]
and nondegenerate:
\[
    \forall v \in T_{x}M,\ \iota_{v} \omega =0 \Rightarrow v =0.
\] 
If $\omega$ is an $n$-plectic form on $M$ we call the pair $(M,\omega)$ 
a {\bf multisymplectic manifold}, or \textbf{n}-{\bf plectic manifold}.
\end{definition}
The name `$n$-plectic' was chosen so that a 1-plectic
structure is a symplectic structure. 

An $n$-plectic structure induces an injective map from the
space of vector fields on $M$ to the space of $n$-forms on $M$. This leads
us to the following definition:

\begin{definition}[\cite{Baez:2008bu}] \label{hamiltonian}
Let $(M,\omega)$ be an $n$-plectic manifold.  An $(n-1)$-form $\alpha$
is {\bf Hamiltonian} iff there exists a vector field $v_\alpha \in \X(M)$ such that
\[
d\alpha= -\ip{\alpha} \omega.
\]
We say $v_\alpha$ is the {\bf Hamiltonian vector field} corresponding to $\alpha$. 
The set of Hamiltonian $(n-1)$-forms and the set of Hamiltonian vector
fields on an $n$-plectic manifold are both vector spaces and are denoted
as $\ham{n-1}$ and $\VectH \left(M \right)$, respectively.
\end{definition}

The Hamiltonian vector field $v_\alpha$ is unique if it exists, but
there may be $(n-1)$-forms having no Hamiltonian vector field.  Note
that if $\alpha \in \Omega^{n-1}(M)$ is closed, then it is Hamiltonian
and its Hamiltonian vector field is the zero vector field. 
 
An elementary, yet important, fact is that the flow of a Hamiltonian
vector field preserves the $n$-plectic structure.
\begin{lemma}[\cite{Baez:2008bu}]\label{L_thm}
If $v_{\alpha}$ is a Hamiltonian vector field, then $\L_{v_{\alpha}}\omega =0$.
\end{lemma}
\begin{proof}
\[
\L_{v_{\alpha}} \omega = d \ip{\alpha} \omega+ \ip{\alpha} d \omega
=-d d \alpha =0
\]
\end{proof}

We now define a bracket on $\ham{n-1}$ that generalizes the Poisson
bracket in symplectic geometry. One motivation for considering this
bracket comes from its appearance in multisymplectic formulations of
classical field theories \cite{Helein:2002wf,Kijowski:1973gi}, in
which the usual infinite-dimensional symplectic phase space is
replaced with a finite-dimensional `multi-phase space'.  

\begin{definition}[\cite{Baez:2008bu}]
\label{bracket_def}
Given $\alpha,\beta\in \ham{n-1}$, the {\bf bracket} $\brac{\alpha}{\beta}$
is the $(n-1)$-form given by 
\[  \brac{\alpha}{\beta} = \iota_{v_{\beta}}\iota_{v_{\alpha}}\omega .\]
\end{definition}

When $n=1$, this bracket is the usual Poisson bracket of smooth
functions on a symplectic manifold. These next propositions show that
for $n>1$ the bracket of Hamiltonian forms has several properties in
common with the Poisson bracket in symplectic geometry. However,
unlike the case in symplectic geometry, we see that the bracket
$\blankbrac$ does not need to satisfy the Jacobi identity for $n >1$.

\begin{prop}[\cite{Baez:2008bu}]\label{bracket_prop} Let $\alpha,\beta \in
  \ham{n-1}$ and
$v_{\alpha},v_{\beta}$ be their respective Hamiltonian
vector fields.  The bracket $\brac{\cdot}{\cdot}$ has the following properties:
  \begin{enumerate}
\item{The bracket is skew-symmetric:
\[
\brac{\alpha}{\beta}=-\brac{\beta}{\alpha}.
\]
}
\item{ The bracket of Hamiltonian forms is Hamiltonian:
  \[
    d\brac{\alpha}{\beta} = -\iota_{[v_{\alpha},v_{\beta}]} \omega,
  \]
and in particular we have 
\[     v_{\brac{\alpha}{\beta}} = [v_{\alpha},v_{\beta}]  .\]
}
\end{enumerate}
\end{prop}
\begin{proof} 
The first statement follows from the antisymmetry of $\omega$. To
prove the second statement, we use Lemma \ref{L_thm}:
\begin{align*}
d\brac{\alpha}{\beta} & = d\ip{\beta}\ip{\alpha} \omega\\
&= \left (\L_{v_{\beta}}-
\iota_{v_{\beta}} d \right) \iota_{v_{\alpha}}\omega \\
&=\L_{v_{\beta}} \iota_{v_{\alpha}}\omega + \iota_{v_{\beta}} d  d\alpha \\
&=\iota_{[v_{\beta},v_{\alpha}]}\omega +
\iota_{v_{\alpha}} \L_{v_{\beta}} \omega\\
&=-\iota_{[v_{\alpha},v_{\beta}]}\omega.
\end{align*}
\end{proof}

\begin{prop}[\cite{Baez:2008bu}]\label{no_jacobi}
The bracket $\blankbrac$ satisfies the Jacobi identity up to an exact $(n-1)$-form:
\[
    \brac{\alpha_{1}}{\brac{\alpha_{2}}{\alpha_{3}}} -
    \brac{\brac{\alpha_{1}}{\alpha_{2}}}{\alpha_{3}} 
    -\brac{\alpha_{2}}{\brac{\alpha_{1}}{\alpha_{3}}} =-d\iota(v_{\alpha_{1}}
      \wedge v_{\alpha_{2}} \wedge v_{\alpha_{3}}) \omega.
\]
\end{prop}

A proof of Proposition \ref{no_jacobi} was given by direct computation
in \cite{Baez:2008bu}. However, it also follows from the next lemma. 
We will use this lemma again in the proof of Theorem \ref{main_thm} in
Section \ref{main}.

\begin{lemma}\label{tech_lemma}
If $(M,\omega)$ is an $n$-plectic manifold and $v_{1},\hdots,v_{m} \in
\VectH(M)$ with $m \geq 2$ then
\begin{multline} \label{big_identity}
d \iota(v_{1} \wedge\cdots \wedge v_{m}) \omega = \\(-1)^{m}\sum_{1 \leq i < j \leq
  m} (-1)^{i+j} \iota([v_{i},v_{j}] \wedge v_{1} \wedge \cdots
  \wedge \hat{v}_{i} \wedge \cdots \wedge \hat{v}_{j} \wedge \cdots \wedge v_{m})
\omega. 
\end{multline}
\end{lemma}
\begin{proof}
We proceed via induction on $m$. For $m=2$:
\[d\iota(v_{1} \wedge v_{2})\omega=d\brac{\alpha_{1}}{\alpha_{2}},
\]
 where $\alpha_{1},\alpha_{2}$
are any Hamiltonian $(n-1)$-forms whose Hamiltonian vector
fields are $v_{1},v_{2}$, respectively. Then Proposition
\ref{bracket_prop} implies Eq.\ \ref{big_identity} holds.

Assume Eq.\ \ref{big_identity} holds for $m-1$. 
Since $\iota(v_{1} \wedge\cdots \wedge v_{m})=\iota_{v_{m}}\iota(
v_{1} \wedge\cdots \wedge v_{m-1})$, Eq.\ \ref{Lie} implies:
\begin{equation} \label{step1}
d \iota(v_{1} \wedge\cdots \wedge v_{m})\omega= \L_{v_{m}}\iota(v_{1} \wedge
  \cdots \wedge v_{m-1}) \omega- \iota_{v_{m}} d \iota(v_{1} \wedge
  \cdots \wedge v_{m-1}) \omega. 
\end{equation}
Consider the first term on the right hand side. Using Eq.\
\ref{commutator} we can rewrite it as
\begin{align*}
\L_{v_{m}}\iota(v_{1} \wedge  \cdots \wedge v_{m-1}) \omega &= 
\iota([v_{m},v_{1} \wedge \cdots \wedge v_{m-1}]) \omega \\
& \quad +\iota(v_{1} \wedge \cdots \wedge v_{m-1})
\L_{v_{m}} \omega \\
&=\iota([v_{m},v_{1} \wedge \cdots \wedge v_{m-1}]) \omega,
\end{align*}
where the last equality follows from Lemma \ref{L_thm}.

The definition of the Schouten bracket given in Eq.\
\ref{Schouten} implies
\[
[v_{m},v_{1} \wedge \cdots \wedge v_{m-1}] =\sum_{i=1}^{m-1}
(-1)^{i+1} [v_{m},v_{i}] \wedge v_{1} \wedge \cdots \wedge \hat{v}_{i}
\wedge \cdots \wedge v_{m-1}.
\]
Therefore we have
\begin{align*}
\L_{v_{m}}\iota(v_{1} \wedge  \cdots \wedge v_{m-1}) \omega
&=\iota([v_{m},v_{1} \wedge \cdots \wedge v_{m-1}]) \omega \\
&=\sum_{i=1}^{m-1}
(-1)^{i} \iota([v_{i},v_{m}] \wedge v_{1} \wedge \cdots \wedge \hat{v}_{i}
\wedge \cdots \wedge v_{m-1})\omega.
\end{align*}
Combining this with the second term in Eq.\ \ref{step1} and using the
inductive hypothesis gives
\begin{align*}
\begin{split}
d \iota(v_{1} \wedge\cdots \wedge v_{m}) \omega 
 = \sum_{i=1}^{m-1}
(-1)^{i} \iota([v_{i},v_{m}] \wedge v_{1} \wedge \cdots \wedge \hat{v}_{i}
\wedge \cdots \wedge v_{m-1})\omega   
\end{split}
\\
& \quad -(-1)^{m-1} \sum \limits_{1 \leq i < j \leq
  m-1} (-1)^{i+j} \iota_{v_{m}}\iota([v_{i},v_{j}] \wedge v_{1} \wedge\cdots \\
& \quad \wedge \hat{v}_{i} \wedge \cdots  \wedge \hat{v}_{j} \wedge \cdots \wedge v_{m-1})
\omega \\
&= (-1)^{m}\left (\sum_{i=1}^{m-1}
(-1)^{i+m} \iota([v_{i},v_{m}] \wedge v_{1} \wedge \cdots \wedge \hat{v}_{i}
\wedge \cdots \wedge v_{m-1})\omega \right.  \\
& \quad \left. + \sum \limits_{1 \leq i < j \leq
  m-1} (-1)^{i+j} \iota([v_{i},v_{j}] \wedge v_{1} \wedge \cdots \wedge
  \hat{v}_{i} \wedge \cdots \wedge \hat{v}_{j} \wedge \cdots \wedge v_{m})
\omega \right)\\
&=(-1)^{m}\sum_{1 \leq i < j \leq
  m} (-1)^{i+j} \iota([v_{i},v_{j}] \wedge v_{1} \wedge \cdots \wedge
  \hat{v}_{i} \wedge \cdots \wedge \hat{v}_{j} \wedge \cdots \wedge v_{m})
\omega. 
\end{align*}
\end{proof}

\begin{proof}[Proof of Proposition \ref{no_jacobi}]
Apply Lemma \ref{tech_lemma} with $m=3$, and use the fact 
that $v_{\brac{\alpha_{i}}{\alpha_{j}}}=[v_{\alpha_{i}},v_{\alpha_{j}}]$.
\end{proof}

\section{$L_{\infty}$-algebras}
Proposition \ref{no_jacobi} implies that we should not
expect $\ham{n-1}$ to be a Lie algebra unless $n=1$. However, the
fact that the Jacobi identity is satisfied modulo boundary terms
suggests we consider what are known as strongly homotopy Lie algebras,
or $L_{\infty}$-algebras \cite{Lada-Markl,LS}.

\begin{definition} \label{Linfty} An
{\boldmath $L_{\infty}$}{\bf-algebra} is a graded vector space $L$
equipped with a collection
\[\left \{l_{k} \maps L^{\tensor k} \to L| 1
  \leq k < \infty \right\}\]
of
skew-symmetric linear maps with  $\deg{l_{k}}=k-2$ such that
the following identity holds for $1 \leq m < \infty :$
\begin{align} \label{gen_jacobi}
   \sum_{\substack{i+j = m+1, \\ \sigma \in \Sh(i,m-i)}}
  (-1)^{\sigma}\epsilon(\sigma)(-1)^{i(j-1)} l_{j}
   (l_{i}(x_{\sigma(1)}, \dots, x_{\sigma(i)}), x_{\sigma(i+1)},
   \ldots, x_{\sigma(m)})=0.
\end{align}
\end{definition}

\begin{definition} \label{LnA} A $L_{\infty}$-algebra $(L,\{l_{k} \})$
  is a {\bf Lie} {\boldmath $n$}{\bf -algebra} iff the underlying
  graded vector space $L$ is concentrated in degrees $0,\hdots,n-1$.
\end{definition}
Note that if $(L,\{l_{k} \})$ is a Lie $n$-algebra, then by degree counting $l_{k} =0 $ for $k > n+1$. 

The identity satisfied by the maps in Definition \ref{Linfty}
can be interpreted as a `generalized  Jacobi identity'. 
Indeed, using the notation $d=l_{1}$ and $[\cdot,\cdot] = l_{2}$, Eq.\ \ref{gen_jacobi} implies
\begin{align*}
d^2&=0\\
d [x_{1},x_{2}] &= [dx_{1},x_{2}] + (-1)^{\deg{x_{1}}}[x_{1},dx_{2}].
\end{align*}
Hence the map $l_{1} \maps L \to L$ can be
interpreted as a differential, while the map $l_{2} \maps L\tensor L
\to L$ can be interpreted as a bracket. The bracket is, of course, skew
symmetric:
\[
[x_{1},x_{2}] = -(-1)^{\deg{x_{1}} \deg{x_{2}}} [x_{2},x_{1}],
\]
but does not need to satisfy the usual Jacobi identity. In fact,
Eq.\ \ref{gen_jacobi} implies:
\begin{multline*}
(-1)^{\deg{x_{1}}\deg{x_{3}}}[[x_{1},x_{2}],x_{3}] + (-1)^{\deg{x_{2}} \deg{x_{3}}}[[x_{3},x_{1}],x_{2}] + (-1)^{\deg{x_{1}}\deg{x_{2}}} [[x_{2},x_{3}],x_{1}] \\= 
(-1)^{\deg{x_{1}}\deg{x_{3}}+1} \bigl (dl_{3}(x_{1},x_{2},x_{3}) + l_{3}(dx_{1},x_{2},x_{3}) \\
+ (-1)^{\deg{x_{1}}} l_{3}(x_{1},dx_{2},x_{3})  + (-1)^{\deg{x_{1}} +
  \deg{x_{2}}} l_{3}(x_{1},x_{2},dx_{3})  \bigr ).
\end{multline*}
Therefore one can interpret the traditional Jacobi identity as a
null-homotopic chain map from $L\tensor L \tensor L$ to $L$. The
map $l_{3}$ acts as a chain homotopy and is referred to as the
\textbf{Jacobiator}. Eq.\ \ref{gen_jacobi} also implies that $l_{3}$
must satisfy a coherence condition of its own. From the above
discussion, it is easy to see that a Lie 1-algebra is an ordinary Lie
algebra, while a $L_{\infty}$-algebra with $l_{k} \equiv 0$ for all $k
\geq 3$ is a differential graded Lie algebra.

\section{The Lie $n$-algebra associated to an $n$-plectic manifold} \label{main}
There are several clues that suggest that any $n$-plectic manifold
gives a $L_{\infty}$-algebra. It was shown
in our previous work \cite{Baez:2008bu} that a Lie 2-algebra can be explicitly
constructed from the 2-plectic structure on any 2-plectic
manifold. The underlying chain complex of this Lie 2-algebra is
\[
\cinf(M) \stackrel{d}{\to} \ham{1},
\]
where $d$ is the de Rham differential. This suggests that for an
arbitrary $n$-plectic manifold, we should look for
Lie $n$-algebra structures on the chain complex 
\begin{equation} \label{complex}
\cinf(M) \stackrel{d}{\to} \Omega^{1}(M) \stackrel{d}{\to} \cdots \stackrel{d}{\to} 
\Omega^{n-2}(M) \stackrel{d}{\to} \ham{n-1},
\end{equation}
with the $l_{1}$ map equal to $d$. We denote this complex as $(L,d)$.
It is concentrated in degrees $0,\ldots,n-1$ with
\[
L_{i} =
\begin{cases}
\ham{n-1} & i=0,\\
\Omega^{n-1-i}(M) & 0 < i \leq n-1.\\
\end{cases}
\]

Note that the bracket
$\blankbrac$ given in Definition \ref{bracket_def} induces a
well-defined bracket $\blankbrac^{\prime}$ on the quotient 
\[
\g=\ham{n-1}/d\Omega^{n-2}(M),
\]
where $d\Omega^{n-2}(M)$ is the space of exact $(n-1)$-forms. This is
because the Hamiltonian vector field of an exact $(n-1)$-form is the
zero vector field. It follows from Proposition \ref{no_jacobi} that
$\left (\g,\blankbrac^{\prime} \right)$ is, in fact, a Lie algebra. 

If $M$ is contractible, then the homology of $(L,d)$ is
\begin{align*}
H_{0}(L)&=\g, \\
H_{k}(L)&=0 \quad \text{for $0<k<n-1$},\\
H_{n-1}(L)&=\R.
\end{align*}

Therefore, the augmented complex 
\begin{equation} \label{aug_complex} 0 \to \R \hookrightarrow \cinf(M)
  \stackrel{d}{\to} \Omega^{1}(M) \stackrel{d}{\to} \cdots
  \stackrel{d}{\to} \Omega^{n-2}(M) \stackrel{d}{\to} \ham{n-1}
\end{equation}
is a resolution of $\g$.

Barnich, Fulp, Lada, and Stasheff \cite{Barnich:1997ij} showed
that, in general, if $(C,\delta)$ is a resolution of a
vector space $V \cong H_{0}(C)$ and $C_{0}$ is equipped with a
skew-symmetric map $\tilde{l}_{2} \maps C_{0} \tensor C_{0} \to C_{0}$
that induces a Lie bracket on $V$, then
$\tilde{l}_{2}$ extends to an $L_{\infty}$-structure on $(C,\delta)$.
Hence we have the following proposition:

\begin{prop}\label{contract_thm}
Given a contractible $n$-plectic manifold $(M,\omega)$, there is a $L_{\infty}$-algebra
$(\tilde{L},\{l_{k} \})$ with underlying graded vector space
\[
\tilde{L}_{i} =
\begin{cases}
\ham{n-1} & i=0,\\
\Omega^{n-1-i}(M) & 0 < i \leq n-1,\\
\R & i=n,
\end{cases}
\]
and $l_{1} \maps \tilde{L} \to \tilde{L}$ defined as 
\[
l_{1}(\alpha)=
\begin{cases}
\alpha, & \text{if $\deg{\alpha}=n$} \\
d\alpha & \text{if $\deg{\alpha} \neq n$,}
\end{cases}
\]
and all higher maps  $\left \{l_{k} \maps \tilde{L}^{\tensor k} \to \tilde{L}| 2
  \leq k < \infty \right\}$ are constructed inductively by using the bracket
\[
\blankbrac \maps \tilde{L}_{0} \tensor \tilde{L}_{0} \to \tilde{L}_{0}, \quad
\brac{\alpha_{1}}{\alpha_{2}} = \ip{\alpha_{2}}\ip{\alpha_{1}} \omega,
\]
where $v_{\alpha_{1}},v_{\alpha_{2}}$ are the Hamiltonian vector fields corresponding
to the Hamiltonian forms $\alpha_{1}, \alpha_{2}$. Moreover the maps
$\{ l_{k} \}$ may be constructed so that
\[
l_{k}(\alphadk{k}) \neq 0 \quad \text{only if all $\alpha_{k}$ have
  degree 0},
\]
for $k \geq 2$. 
\end{prop}
\begin{proof}
  The proposition follows from Theorem 7 in the paper by Barnich,
  Fulp, Lada, and Stasheff \cite{Barnich:1997ij}. Since
  for any $n$-plectic manifold,
\[
\brac{\alpha}{d \beta}=0 \quad \forall \alpha \in \ham{n-1} ~ \forall
\beta \in \Omega^{n-2}(M),
\]
the second remark following Theorem 7 in \cite{Barnich:1997ij} implies
that the maps $\{ l_{k}\}$ may be constructed so that they are trivial
when restricted to the positive-degree part of the $k$-th tensor power of $\tilde{L}$.
\end{proof}

For an arbitrary $n$-plectic manifold $(M,\omega)$, Proposition
\ref{contract_thm} guarantees the existence of $L_{\infty}$-algebras
locally. We want, of course, a global result in which the higher
$l_{k}$ maps are explicitly constructed using only the $n$-plectic
structure. Moreover, in our previous work on 2-plectic geometry, we
were able to construct by hand a Lie 2-algebra on a 2-term complex
consisting of functions and Hamiltonian 1-forms. We did not need to
use a 3-term complex consisting of constants, functions, and
Hamiltonian 1-forms. Hence in the general case, we'd expect an
$n$-plectic manifold to give a Lie $n$-algebra whose underlying
complex is $(L,d)$, instead of a Lie $(n+1)$-algebra whose underlying
complex is the $(n+1)$-term complex used in the above proposition.

We can get an intuitive sense for what the maps $l_{k} \maps L^{\tensor k} \to L$
should be by unraveling the identity given in Definition \ref{Linfty}
for small values of $m$ and momentarily disregarding signs and
summations over unshuffles. For example, if $m=2$, then Eq.\
\ref{gen_jacobi} implies that the map $l_{2} \maps L \tensor L \to L$
must satisfy:
\begin{equation} \label{l2}
l_{1}l_{2} + l_{2} l_{1}=0.
\end{equation}
Obviously we want $l_{1}$ to be the de Rham differential and $l_{2}$
to be equal to the bracket $\blankbrac$ when restricted to degree 0
elements:
\[
l_{2}(\alpha_{1},\alpha_{2}) = \pm \iota_{v_{\alpha_{2}}}
\iota_{v_{\alpha_{1}}} \omega=\brac{\alpha_{1}}{\alpha_{2}}\quad \forall \alpha_{i} \in
L_{0}=\ham{n-1}.
\]
Now consider elements of degree 1. For example, if $\alpha \in
L_{0}$ and $\beta \in
L_{1}=\Omega^{n-2}(M)$, then
$l_{2}(\alpha,d\beta)=\brac{\alpha}{d\beta} =0$.
Therefore Eq.\ \ref{l2} implies
\[
dl_{2}(\alpha,\beta)=  l_{1} l_{2}(\alpha,\beta)= 0.
\]
Hence, when restricted to elements of degree 1, $l_{2}(\alpha,\beta)$
must be a closed $(n-2)$-form. We will choose this closed
form to be 0. In fact, we will choose $l_{2}$ to vanish on all elements with degree
$>0$, since, in general, we want the $L_{\infty}$ structure to 
only depend on the de Rham differential and the $n$-plectic structure.

Now suppose $l_{2}$ is defined as above and
let $m=3$.  Then Eq.\ \ref{gen_jacobi} implies:
\begin{equation}\label{l3}
l_{1} l_{3} + l_{2} l_{2} + l_{3} l_{1} =0.
\end{equation}
On degree 0 elements, $l_{1}=0$. Therefore it's clear from Proposition
\ref{no_jacobi} that the map $l_{3} \maps L^{\tensor 3} \to L$ when
restricted to degree 0 elements must be
\[
l_{3}(\alpha_{1},\alpha_{2},\alpha_{3}) = \pm \iota(v_{\alpha_{1}}
\wedge v_{\alpha_{2}} \wedge v_{\alpha_{3}}) \omega,\] where
$v_{\alpha_{i}}$ is the Hamiltonian vector field associated to
$\alpha_{i}$.  Now consider a degree 1 element of $L \tensor L \tensor
L$, for example: $\alpha_{1} \tensor \alpha_{2} \tensor \beta \in 
\ham{n-1}\tensor \ham{n-1} \tensor \Omega^{n-2}(M)$. Since
$l_{3}(\alpha_{1},\alpha_{2},d\beta)= \pm \iota(v_{\alpha_{1}}
\wedge v_{\alpha_{2}} \wedge v_{d\beta}) \omega=0$,
and $l_{2}$ vanishes on 
the positive-degree part of the $k$-th tensor power of $L$,
Eq.\ \ref{l3} holds if and only if
\[
dl_{3}(\alpha_{1},\alpha_{2},\beta)=0.
\]
Hence, when restricted to elements of degree 1, $l_{3}(\alpha_{1},\alpha_{2},\beta)$
must be a closed $(n-2)$-form. Again, we will choose this closed
form to be 0 by forcing $l_{3}$ to vanish on all elements with degree
$>0$. 

Observations like these bring us to our main theorem.  In general, we
will define the maps $l_{k} \maps L^{\tensor k} \to L$ on degree zero
elements to be completely specified (up to sign) by the $n$-plectic
structure $\omega$:
\[
l_{k}(\alphadk{k}) = \pm \iota(\vk{k}) \omega \quad \text{if $\deg{\alphak{k}}=0$},
\]
and trivial otherwise:
\[
l_{k}(\alphadk{k}) = 0 \quad \text{if $\deg{\alphak{k}} > 0$}.
\]

\begin{theorem} \label{main_thm}
Given a $n$-plectic manifold $(M,\omega)$, there is a Lie $n$-algebra
$\Lie(M,\omega)=(L,\{l_{k} \})$ with underlying graded vector space 
\[
L_{i} =
\begin{cases}
\ham{n-1} & i=0,\\
\Omega^{n-1-i}(M) & 0 < i \leq n-1,
\end{cases}
\]
and maps  $\left \{l_{k} \maps L^{\tensor k} \to L| 1
  \leq k < \infty \right\}$ defined as
\[ 
l_{1}(\alpha)=d\alpha,
\]
if $\deg{\alpha}>0$ and
\begin{multline}
l_{k}(\alphadk{k}) =\\
\begin{cases}
0 & \text{if $\deg{\alphak{k}} > 0$}, \\
(-1)^{\frac{k}{2}+1} \iota(\vk{k}) \omega  & \text{if
  $\deg{\alphak{k}}=0$ and $k$ even},\\
(-1)^{\frac{k-1}{2}}\iota(\vk{k}) \omega  & \text{if
  $\deg{\alphak{k}}=0$ and $k$ odd},
\end{cases}
\end{multline}
for $k>1$, where $v_{\alpha_{i}}$ is the unique Hamiltonian vector field
associated to $\alpha_{i} \in \ham{n-1}$.
\end{theorem}

\begin{proof}[Proof of Theorem \ref{main_thm}]
We begin by showing the maps $\{l_{k}\}$
are well-defined skew symmetric maps with $\deg{l_{k}}=k-2$.
If $\alphak{k} \in L^{\tensor \bullet}$ has
degree 0, then for all $\sigma \in \Sn_{k}$ the antisymmetry of
$\omega$ implies
\[
l_{k}(\alphasdk{k})=(-1)^{\sigma} l_{k}(\alphadk{k}).
\]
Since for each $i$, we have $\deg{\alpha_{i}}=0$, it follows that
$\epsilon(\sigma)=1$. Hence $l_{k}$ is skew symmetric and
well-defined. Since $\iota(\vk{k}) \omega \in
\Omega^{n+1-k}(M)=L_{k-2}$, we have $\deg{l_{k}}=k-2$.
We also have, by construction, $l_{k} = 0 $ for $k>n+1$.

Now we prove the maps satisfy Eq.\ \ref{gen_jacobi} in Definition
\ref{Linfty}. If $m=1$, then it is satisfied since $l_{1}$
is the de Rham differential. If $m=2$, then a direct calculation shows
\[
l_{1}(l_{2}(\alpha_{1},\alpha_{2})) = l_{2}(l_{1}(\alpha_{1}),\alpha_{2}) + (-1)^{\deg{\alpha_{1}}}l_{2}(\alpha_{1},l_{1}(\alpha_{2})).
\]
Let $m > 2$. We will regroup the summands
in Eq.\ \ref{gen_jacobi} into two separate sums depending on the value
of the index $j$ and show that each of these is zero, thereby proving the
theorem.

We first consider the sum of the terms with $2 \leq j \leq m-2$:
\begin{equation} \label{term1}
\sum_{j=2}^{m-2}
   \sum_{\sigma \in \Sh(i,m-i)} \negthickspace \negthickspace \negthickspace
   (-1)^{\sigma}\epsilon(\sigma)(-1)^{i(j-1)} l_{j}
   (l_{i}(\alpha_{\sigma(1)}, \dots, \alpha_{\sigma(i)}), \alpha_{\sigma(i+1)},
   \ldots, \alpha_{\sigma(m)}).
\end{equation}
In this case we claim that for all $\sigma \in \Sh(i,m-i)$ we have 
\[
l_{j}(l_{i}(\alphasdk{i}),\alpha_{\sigma(i+1)},\hdots,\alpha_{\sigma(m)})=0.
\]
Indeed, if there exists an unshuffle such that the above
equality did not hold, then the definition of $l_{j} \maps L^{\tensor
  j} \to L$ implies
\[
\deg{l_{i}(\alphasdk{i}) \tensor \alpha_{\sigma(i+1)} \tensor \cdots \tensor \alpha_{\sigma(m)}}=0,
\] 
which further implies
\begin{equation} \label{step2}
\deg{l_{i}(\alphasdk{i})}=\deg{\alphask{i}} + i -2 =0.
\end{equation}
By assumption, $l_{i}(\alphasdk{i})$ must be non-zero and $j < m-1$
implies $i>1$. Hence we must have $\deg{\alphask{i}}=0$ and
therefore, by Eq.\ \ref{step2}, $i=2$. But this
implies $j=m-1$, which contradicts our bounds on $j$. So no such
unshuffle could exist, and therefore the sum (\ref{term1}) is zero.  

We next consider the sum of the terms $j=1$, $j=m-1$, and $j=m$:
\begin{equation} \label{the_sum}
\begin{split}
l_{1}( l_{m}(\alphadk{m})) + \sum_{\sigma \in \Sh(2,m-2)} 
(-1)^{\sigma}
\epsilon(\sigma)
l_{m-1}(l_{2}(\alpha_{\sigma(1)},\alpha_{\sigma(2)}),\alpha_{\sigma(3)},\hdots,\alpha_{\sigma(m)})
 \\
+ \sum_{\sigma \in \Sh(1,m-1)}  \negthickspace \negthickspace \negthickspace
(-1)^{\sigma}
\epsilon(\sigma)  (-1)^{m-1} 
l_{m}(l_{1}(\alpha_{\sigma(1)}),\alpha_{\sigma(2)},\hdots,\alpha_{\sigma(m)}).
\end{split}
\end{equation}
Note that if $\sigma \in \Sh(1,m-1)$ and $\deg{l_{1}(\alpha_{\sigma(1)})} >0$, then
\[
l_{m}(l_{1}(\alpha_{\sigma(1)}),\alpha_{\sigma(2)},\hdots,\alpha_{\sigma(m)})=0
\]
by definition of the map $l_{m}$. On the other hand, 
if $\deg{l_{1}(\alpha_{\sigma(1)})} =0$, then
$l_{1}(\alpha_{\sigma(1)})=d \alpha_{\sigma(1)}$ is Hamiltonian and
its Hamiltonian vector field is the zero vector field. Hence the third
term in (\ref{the_sum}) is zero. 

Since the map $l_{2}$ is degree 0, we only need to consider the
first two terms of (\ref{the_sum}) in the case when $\deg{\alphak{m}}
=0$. For the first term we have:
\[
l_{1}( l_{m}(\alphadk{m}))=
\begin{cases}
(-1)^{\frac{m}{2}+1} d\iota(\vk{m}) \omega & \text{if $m$ even},\\
(-1)^{\frac{m-1}{2}}d\iota(\vk{m}) \omega  & \text{if $m$ odd}.
\end{cases}
\]
Now consider the second term. If $\alpha_{i},\alpha_{j} \in \ham{n-1}$ are Hamiltonian
$(n-1)$-forms then by Definition \ref{bracket_def},
$l_{2}(\alpha_{i},\alpha_{j})=\brac{\alpha_{i}}{\alpha_{j}}$.
By Proposition \ref{bracket_prop}, $l_{2}(\alpha_{i},\alpha_{j})$ is
Hamiltonian and its Hamiltonian vector field is $v_{\brac{\alpha_{i}}{\alpha_{j}}}=[v_{\alpha_{i}},v_{\alpha_{j}}]$.
Therefore for  $\sigma \in \Sh(2,m-2)$, we have
\begin{multline*}
l_{m-1}(l_{2}(\alpha_{\sigma(1)},\alpha_{\sigma(2)}),\alpha_{\sigma(3)},\hdots,\alpha_{\sigma(m)})=\\
\begin{cases}
(-1)^{\frac{m}{2}-1}\iota([v_{\alpha_{\sigma(1)}},v_{\alpha_{\sigma(2)}}] \wedge
  \cdots \wedge v_{\alpha_{\sigma(m)}}) \omega & \text{if $m$ even},\\
(-1)^{\frac{m+1}{2}}\iota([v_{\alpha_{\sigma(1)}},v_{\alpha_{\sigma(2)}}] \wedge
  \cdots \wedge v_{\alpha_{\sigma(m)}}) \omega  & \text{if $m$ odd}.
\end{cases}
 \end{multline*}
Since each $\alpha_{i}$ is degree 0, we can rewrite the sum over
$\sigma \in \Sh(2,m-2)$ as
\begin{multline*}
\sum_{\sigma \in \Sh(2,m-2)}
(-1)^{\sigma}
\epsilon(\sigma)
l_{m-1}(l_{2}(\alpha_{\sigma(1)},\alpha_{\sigma(2)}),\alpha_{\sigma(3)},\hdots,\alpha_{\sigma(m)}) =\\
\sum_{1\leq i<j \leq m} (-1)^{i+j-1} l_{m-1}(l_{2}(\alpha_{i},\alpha_{j}),\alpha_{1}, \alpha_{2},\hdots,
\hat{\alpha}_{i},\hdots,\hat{\alpha}_{j},\hdots,\alpha_{m}).
\end{multline*}
Therefore, if $m$ is even, the sum (\ref{the_sum}) becomes
\begin{multline*}
(-1)^{\frac{m}{2}+1} d\iota(\vk{m}) \omega  
+ (-1)^{\frac{m}{2}} \sum_{1\leq i<j \leq m} (-1)^{i+j} \iota([v_{\alpha_{i}},v_{\alpha_{j}}]
  \wedge v_{\alpha_{1}} \\ \wedge  \cdots 
\wedge \hat{v}_{\alpha_{i}} \wedge \cdots \wedge
  \hat{v}_{\alpha_{j}} \wedge \cdots \wedge v_{\alpha_{m}}) \omega
\end{multline*}
and, if $m$ is odd:
\begin{multline*}
(-1)^{\frac{m-1}{2}} d\iota(\vk{m}) \omega  
+ (-1)^{\frac{m-1}{2}} \sum_{1\leq i<j \leq m} (-1)^{i+j} \iota([v_{\alpha_{i}},v_{\alpha_{j}}]
  \wedge v_{\alpha_{1}} \\ \wedge
  \cdots \wedge \hat{v}_{\alpha_{i}} \wedge \cdots  \wedge
  \hat{v}_{\alpha_{j}} \wedge \cdots \wedge v_{\alpha_{m}}) \omega.
\end{multline*}
It then follows from Lemma \ref{tech_lemma} that, in either case, (\ref{the_sum}) is zero.
\end{proof}

It is clear that in the $n=1$ case, $\Lie(M,\omega)$ is the underlying
Lie algebra of the usual Poisson algebra of smooth functions on a
symplectic manifold. In the $n=2$ case, $\Lie(M,\omega)$ is the Lie
2-algebra obtained in our previous work with Baez and Hoffnung
\cite{Baez:2008bu}.  

Note that the equality
\[
    d\brac{\alpha}{\beta} = -\iota_{[v_{\alpha},v_{\beta}]} \omega
\]
given in Proposition \ref{bracket_prop} implies the existence of a
bracket-preserving chain map
\[
\phi \maps \Lie(M,\omega) \to \VectH \left(M \right),
\]
which in degree 0 takes a Hamiltonian $(n-1)$-form 
$\alpha$ to its vector field $v_{\alpha}$. 
Here we consider the Lie algebra of Hamiltonian vector fields as a Lie
1-algebra whose underlying complex is concentrated in degree 0:
\[
\ldots \to 0 \to 0 \to \VectH \left(M \right).
\]
Hence $\phi$ is trivial in all higher degrees. In light of
Theorem \ref{main_thm}, $\phi$ becomes a strict morphism of
$L_{\infty}$-algebras. (See the paper by Lada and Markl \cite{Lada-Markl} for the
definition of  $L_{\infty}$-algebra morphisms).

\section{The dg Leibniz algebra associated to an $n$-plectic manifold}
In symplectic geometry, every
function $f \in \cinf(M)$ is Hamiltonian.  We also have the equality:
\begin{equation} \label{poisson}
\brac{f}{g}= \ip{f}dg= \L_{v_{f}}g
\end{equation}
for all $f,g \in \ham{0}=\cinf(M)$. Hence $\brac{f}{\cdot}$ is a
degree zero derivation on $\ham{0}$, which makes
$(\ham{0},\blankbrac)$ a Poisson algebra. In general, for $n >1$, 
an equality such as Eq.\ \ref{poisson} does not hold, and
Hamiltonian forms are obviously not closed under wedge product.
Therefore, we shouldn't expect the Lie $n$-algebra $\Lie(M,\omega)$ to
behave like a Poisson algebra. But we do have the following simple lemma:
\begin{lemma} \label{der_lemma}
Let $(M,\omega)$ be an $n$-plectic manifold. If $\alpha,\beta \in
\ham{n-1}$ are Hamiltonian forms, then
\[
\L_{v_{\alpha}} \beta= \brac{\alpha}{\beta} + d \ip{\alpha} \beta.
\]
\end{lemma}
\begin{proof}
Definitions \ref{hamiltonian} and \ref{bracket_def} imply:
\begin{align*}
\L_{v_{\alpha}} \beta &= \ip{\alpha} d\beta + d\ip{\alpha} \beta \\
&= -\ip{\alpha} \ip{\beta}\omega + d\ip{\alpha} \beta \\
&=\brac{\alpha}{\beta} + d \ip{\alpha} \beta.
\end{align*}
\end{proof}

Lemma \ref{der_lemma} suggests that we interpret the $(n-1)$-form
$\L_{v_{\alpha}}\beta$ as a type of bracket on $\ham{n-1}$, equal to
the bracket given in Definition \ref{bracket_def} modulo boundary
terms. To this end, we consider an algebraic structure known as a
differential graded (dg) Leibniz algebra. 

\begin{definition} \label{dg_leibniz}
A {\bf differential graded Leibniz algebra}
$(L,\delta,\cbrac{\cdot}{\cdot})$ is a graded vector
space $L$ equipped with a degree -1 linear map $\delta \maps L \to
L$ and a degree 0 bilinear map $\cbrac{\cdot}{\cdot} \maps L \tensor L
\to L$ such that the following identities hold:
\begin{gather}
\delta \circ \delta =0 \\
\delta \cbrac{x}{y} = \cbrac{\delta x}{y} +
(-1)^{\deg{x}}\cbrac{x}{\delta y} \label{bracket_der}\\
\cbrac{x}{\cbrac{y}{z}}= \cbrac{\cbrac{x}{y}}{z} + (-1)^{\deg{x}\deg{y}}
\cbrac{y}{\cbrac{x}{z}}\label{leib_jacobi},
\end{gather}
for all $x,y,z \in L$.
\end{definition}
In the literature, dg Leibniz algebras are also called dg Loday algebras.
This definition presented here is equivalent to the one given by Ammar and Poncin
\cite{Ammar:2008}. Note that the second condition given in the
definition above can be interpreted as the Jacobi identity. Hence 
if the bilinear map $\cbrac{\cdot}{\cdot}$ is skew-symmetric,
then a dg Leibniz algebra is a DGLA.

We now show that every $n$-plectic manifold gives a dg Leibniz
algebra.
\begin{prop} \label{n-plectic_Leibniz}
Given an $n$-plectic manifold $(M,\omega)$, there is a differential
graded Leibniz algebra $\Leib(M,\omega)=(L,\delta,\cbrac{\cdot}{\cdot})$ with underlying
graded vector space 
\[
L_{i} =
\begin{cases}
\ham{n-1} & i=0,\\
\Omega^{n-1-i}(M) & 0 < i \leq n-1,
\end{cases}
\]
and maps $ \delta \maps L \to L$, $\cbrac{\cdot}{\cdot} \maps L
\tensor L \to L$ defined as
\[
\delta(\alpha)=d \alpha,
\]
if $\deg{\alpha} > 0$ and
\[
\cbrac{\alpha}{\beta}=
\begin{cases}
\L_{v_{\alpha}}\beta & \text{if $\deg{\alpha}=0$,}\\
0 & \text{if $\deg{\alpha} >0$,}
\end{cases}
\]
where $v_{\alpha}$ is the Hamiltonian vector field associated to $\alpha$.
\end{prop}
\begin{proof}
If $\alpha,\beta \in L_{0}=\ham{n-1}$ are Hamiltonian, then Lemma \ref{der_lemma}
implies
$d\cbrac{\alpha}{\beta}=d\brac{\alpha}{\beta}=-\iota_{[v_{\alpha},v_{\beta}]} \omega$.
Hence $\cbrac{\alpha}{\beta}$ is Hamiltonian. For $\deg{\beta} > 0$,
we have $\deg{\L_{v_{\alpha}}\beta}=\deg{\beta}$, since the Lie
derivative is a degree zero derivation. Hence $\cbrac{\cdot}{\cdot}$
is a bilinear degree 0 map. 

We next show that Eq.\ \ref{bracket_der} of Definition \ref{dg_leibniz}
holds. If $ \deg{\alpha} > 1$, then it holds
trivially. If $\deg{\alpha} =1$, then
$\cbrac{\alpha}{\beta}=\cbrac{\alpha}{\delta\beta}=0$ for all $\beta \in L$
by definition, and $\cbrac{\delta \alpha}{\beta}=0$ since the
Hamiltonian vector field associated to $d\alpha$ is zero. If
$\deg{\alpha}=0$ and $\deg{\beta}=0$, then
$\deg{\cbrac{\alpha}{\beta}}=0$. Hence all terms in
\eqref{bracket_der} vanish by definition. The last case to
consider is
$\deg{\alpha}=0$ and $\deg{\beta} > 0$. We have 
\[
\delta \cbrac{\alpha}{\beta}= d \L_{v_{\alpha}}\beta =
\L_{v_{\alpha}} d\beta =\cbrac{\alpha}{\delta \beta}.
\]

Finally, we show the Jacobi identity \eqref{leib_jacobi}
holds. Let $\alpha, \beta, \gamma \in L$.  Then the left hand side of
\eqref{leib_jacobi} is $\cbrac{\alpha}{\cbrac{\beta}{\gamma}}$, while
the right hand side is $\cbrac{\cbrac{\alpha}{\beta}}{\gamma} +
(-1)^{\deg{\alpha}\deg{\beta}}\cbrac{\beta}{\cbrac{\alpha}{\gamma}}$. Note equality holds trivially
if $\deg{\alpha} >0$ or $\deg{\beta} >0$. Otherwise, we use the
identity
\[
\L_{[v_{1},v_{2}]} = \L_{v_{1}} \L_{v_{2}} - \L_{v_{2}}\L_{v_{1}},
\]
and the fact that $d\cbrac{\alpha}{\beta}=-\iota_{[v_{\alpha},v_{\beta}]} \omega$
to obtain the following equalities:
\begin{align*}
\cbrac{\alpha}{\cbrac{\beta}{\gamma}} &= \L_{v_{\alpha}}
\L_{v_{\beta}} \gamma\\
&=\L_{[v_{\alpha},v_{\beta}]}\gamma + \L_{v_{\beta}}
\L_{v_{\alpha}} \gamma\\
&=\cbrac{\cbrac{\alpha}{\beta}}{\gamma} + 
\cbrac{\beta}{\cbrac{\alpha}{\gamma}}.
\end{align*}
\end{proof}
One interesting aspect of the dg Leibniz structure is that it
interprets the bracket of Hamiltonian $(n-1)$-forms geometrically as
the change of an observable along the flow of a Hamiltonian vector
field. Leibniz algebras, in fact, naturally arise in a
variety of geometric settings e.g.\  in Courant algebroid theory
and, more generally, in the derived bracket
formalism \cite{KosSchwarz:2004}.  It would be interesting to compare
$\Leib(M,\omega)$ to the Leibniz algebras that appear in these other
formalisms.

\section{Concluding remarks and open questions}

\subsection{Applications of Theorem \ref{main_thm}}
We wish to remark that Theorem \ref{main_thm} implies that one can
assign a $L_{\infty}$-algebra to each of the multisymplectic manifolds
mentioned in the introduction. These algebraic structures may be of
interest in their own right.  For example, if $(M,\omega)$ is a
compact, connected, oriented $(n+1)$-dimensional manifold equipped
with a volume form $\omega$, then Zambon \cite{Zambon:2010ka} showed
that the isomorphism class of the Lie $n$-algebra $\Lie(M,\omega)$ is
independent of the choice of $\omega$ and therefore only depends on
the manifold $M$.

Another example comes from representation theory and quantum groups.
Given a simple finite-dimensional Lie algebra $\g$ of type $ADE$, one
can construct certain hyper-K\"{a}hler manifolds known as
`Nakajima quiver varieties' \cite{Nakajima:1994}. These can be
used to study the finite-dimensional representations of the quantum
enveloping algebra of the affine Lie algebra corresponding to $\g$
\cite{Nakajima:2000}. As mentioned earlier, every hyper-K\"{a}hler
manifold is 3-plectic. Therefore we can associate a Lie 3-algebra to
any Nakajima quiver variety. It would be interesting to see how these
Lie 3-algebras are related to the representation-theoretic structures
encoded in these varieties.

\subsection{Lie $n$-algebras and dg Leibniz algebras}
By extending the work of Baez and Crans \cite{HDA6}, Roytenberg
\cite{Roytenberg_L2A} developed what are known as 2-term weak
$L_{\infty}$-algebras, or `weak Lie 2-algebras'. In a weak Lie
2-algebra, the skew symmetry condition on the maps given in Definition
\ref{Linfty} is relaxed. In particular, the bilinear map $l_{2} \maps
L \tensor L \to L$ is skew-symmetric only up to homotopy. 
This homotopy must satisfy a coherence condition, as well as compatibility
conditions with the homotopy that controls the failure of the Jacobi identity.
Lie 2-algebras in the sense of Definition \ref{LnA} 
are weak Lie 2-algebras that satisfy skew-symmetry on the
nose. They are called `semi-strict Lie 2-algebras' in this context,
since the Jacobi identity may still fail to hold.
Weak Lie 2-algebras that satisfy a Jacobi identity of the form
\[ 
[x,[y,z]]-[[x,y],z] - [y,[x,z]] =0,
\]
but not necessarily satisfy the skew-symmetry condition,
are called `hemi-strict Lie 2-algebras'. In fact, any
hemi-strict Lie 2-algebra is a 2-term dg Leibniz algebra.

Given an $n$-plectic manifold $(M,\omega)$, it is easy to show that
the bracket of degree 0 elements in the dg Leibniz algebra $\Leib(M,\omega)$ is
skew-symmetric up to an exact $(n-1)$ form:
\[
\cbrac{\alpha}{\beta} + \cbrac{\beta}{\alpha} = d \left( \ip{\alpha}
  \beta + \ip{\beta} \alpha \right).
\]
When $n=2$, we showed in our previous work with Baez and Hoffnung
\cite{Baez:2008bu} that $\Leib(M,\omega)$ is a hemi-strict Lie
2-algebra, and the map
\begin{gather*}
L_{0} \tensor L_{0} \to L_{1}\\
\alpha \tensor \beta \mapsto \left( \ip{\alpha}
  \beta + \ip{\beta} \alpha \right)
\end{gather*}
is the relevant homotopy. We then showed that Roytenberg's definition of morphism is
flexible enough to allow the identity map
on the underlying chain complexes to lift to an actual isomorphism:
\[
\Lie(M,\omega) \cong \Leib(M,\omega) \quad (\text{for} ~ n=2),
\]
in the category of weak Lie 2-algebras. In general, we would like to
conjecture that some sort of equivalence such as this holds for
$n>2$. Unfortunately, it isn't clear in what category this should
occur. Indeed, developing a theory of weak Lie $n$-algebras is an open
problem.  Perhaps by studying the relationships between the structures
specifically on $\Lie(M,\omega)$ and $\Leib(M,\omega)$ for arbitrary
$n$ one could get a sense of what explicit coherence conditions would
be needed to give a good definition.

On the other hand, there are structures known as
`Loday-$\mathbf{\infty}$ algebras' (or sh Leibniz algebras)
\cite{Ammar:2008,Uchino:2009} that generalize the definition of an
$L_{\infty}$-algebra by, again, relaxing the skew symmetry condition
on the maps $\{l_{k} \}$. However, this time the skew symmetry is not
required to hold up to homotopy.  Hence any dg Leibniz algebra is a
Loday-$\infty$ algebra. Any $L_{\infty}$-algebra is as well. Therefore
there may be an isomorphism between $\Lie(M,\omega)$ and
$\Leib(M,\omega)$ in this category for $n \geq 2$.

\subsection{Multisymplectic geometry and the Gelfand-Dickey-Dorfman formalism}
In many formalisms for classical field theory, fields are considered
to be sections of a vector bundle. The observables are represented by
`local functionals' which are evaluated on sections with compact
support. Local functionals are integrals whose integrands are
functions that depend only on the fields and a finite number of their
derivatives. Such functions are called `local functions'. As usual, one can study
the time evolution of the field theory by defining a Poisson bracket
on the local functionals.  However, there is an advantage to working
with local functions directly since they are smooth functions on a
finite-dimensional space (specifically, a finite jet bundle). The
trade-off with this approach is that there is not a one-to-one
correspondence between local functionals and local functions. One has
to consider equivalence classes of local functions modulo total
divergences \cite{Dickey:1997}. Roughly, the formalism developed by Gelfand, Dickey, and
Dorfman \cite{Gelfand:1976,Gelfand:1979} involves considering the Poisson bracket on local functionals
as being induced by a skew-symmetric bracket on local functions which
is a Lie bracket only up to a total divergence. Barnich, Fulp, Lada,
and Stasheff \cite{Barnich:1997ij} showed, using the variational
bicomplex \cite{Anderson:1994}, that such a bracket gives rise to an $L_{\infty}$-algebra.

There are conceptual similarities between the $L_{\infty}$-algebras
arising in multisymplectic geometry and those
constructed by Barnich, Fulp, Lada, and Stasheff. For example,
they both naturally appear when one attempts to treat the observables
of a classical field theory within a finite-dimensional setting.
Historically, it appears that multisymplectic geometry has
developed, for the most part, independently from those formalisms
which use the variational bicomplex. However Bridges, Hydon,
and Lawson \cite{Bridges:2010} have recently reinterpreted the
multisymplectic formalism using the variational bicomplex and it may
be possible to use their results to directly compare the
$L_{\infty}$-algebras that arise in
multisymplectic geometry with those found in the Gelfand-Dickey-Dorfman formalism.

\subsection{Other generalizations of Poisson brackets}
In Nambu mechanics \cite{Gautheron:1996,Nambu:1973}, one considers a
manifold $M$ equipped with an $n$-ary skew-symmetric bracket
$\{\ldots\}$ on the algebra of smooth functions satisfying
the identity:
\[
\{f_{1},\ldots,f_{n-1},\{g_{1},\ldots,g_{n}\}\}=\sum^{n}_{i=1}\{g_{1},\ldots,\{f_{1},\ldots,f_{n-1},g_{i}\},\ldots,g_{n}\},
\]
for all $f_{i},g_{i} \in \cinf(M)$. The vector space $\cinf(M)$ equipped with such a bracket is an example of
an `$n$-Lie algebra' \cite{Filippov:1985}. These structures are quite different
from the Lie $n$-algebras considered here. There are, however, at least
some elementary relationships between the Nambu and $n$-plectic formalisms.
For example, an $n$-plectic form on a
manifold of dimension $n+1$ determines
a dual multivector field $\pi$ of degree $n+1$. This multivector field gives an $(n+1)$-Lie bracket:
\[
\{f_{1},\ldots,f_{n+1}\}=\pi(df_{1},\ldots,df_{n+1}).
\]
(See Theorem 1 in \cite{Gautheron:1996}.)
The $n$-plectic form also determines the graded
skew-symmetric map $l_{n+1}\maps L^{\tensor n+1} \to L$ defined in
Theorem \ref{main_thm} as part of the structure of $\Lie(M,\omega)$.
However, by definition, the restriction of $l_{n+1}$ to $\cinf(M)$ is trivial if $n >1$. 

The grading of the underlying vector space plays a key role in the
theory of $L_{\infty}$-algebras and, in particular, the Lie
$n$-algebras constructed in the present work.
The $n$-Lie algebras of the form $(\cinf(M),\{\ldots\})$,
on the other hand, are trivially graded structures. In fact, it has been demonstrated
that $n$-Lie algebras can be understood as ``ungraded'' analogues of 
Lie $n$-algebras \cite{Iuliu-Lazaroiu:2009}.

Finally, we mention that other algebraic structures have been considered
within the multisymplectic formalism which incorporate Hamiltonian forms
of arbitrary degree. (See, for example, \cite{Cantrijn:1996} and
\cite{Forger:2002ak}.)  These forms are related to the $n$-plectic
structure via Hamiltonian multivector fields. It may be worthwhile to
investigate whether such Hamiltonian forms can be incorporated into
the Lie $n$-algebra and dg-Leibniz structures considered here.

\section*{Acknowledgments}
We thank John Baez, Jim Stasheff, and Marco Zambon for helpful comments and discussions.
We also thank the referees for their suggestions and remarks.

\end{document}